\documentclass{amsart}[11pt]

\newtheorem{theorem}{Theorem}[section]
\newtheorem{lemma}[theorem]{Lemma}
\newtheorem{trditev}[theorem]{Proposition}

\theoremstyle{definition}

\theoremstyle{remark}
\newtheorem{remark}[theorem]{Remark}

\numberwithin{equation}{section}

\usepackage{amssymb}
\usepackage{pst-node}

\def\max{\mathop{\rm max}\nolimits}

\def\Span{\mathop{\rm Span}\nolimits}
\def\Trace{\mathop{\rm Tr}\nolimits}

\def\diag{\mathop{\rm diag}\nolimits}
\def\Rea{\mathop{\rm Re}\nolimits}
\def\Minor{\mathop{\rm Minor}\nolimits}
\def\Ima{\mathop{\rm Im}\nolimits}

\parskip=\smallskipamount

\begin{document}

\noindent
This is a pre-print of an article published in Periodica Mathematica Hungarica. The final authenticated version is available online at:\\ https://doi.org/10.1007/s10998-018-0238-z. 

\vspace{1cm}

\title[]
{On regular Stein neighborhoods of a union of two maximal totally real subspaces in $\mathbb{C}^n$}
\author{Tadej Star\v{c}i\v{c}}
\address{Faculty of Education, University of Ljubljana, Kardeljeva Plo\v{s}\v{c}ad 16, 1000 Lju\-blja\-na, Slovenia and Institute of Mathematics, Physics and Mechanics, Jadranska 19, 1000
Ljubljana, Slovenia}
\email{tadej.starcic@pef.uni-lj.si}
\date{September 16, 2016}


\keywords{Stein neighborhoods, totally real subspaces, strongly pseudoconvex domains, strong deformation retraction\\
\indent Research supported by grants P1-0291 and J1-5432 from ARRS, Republic of Slovenia.}

\begin{abstract}
We present a construction of regular Stein neighborhoods of a union of maximally totally real subspaces 
$M=(A+iI)\mathbb{R}^n$ and $N=\mathbb{R}^n$ in $\mathbb{C}^n$, provided that the entries of a real $n \times n$ 
matrix $A$ are sufficiently small. Our proof is based on a local construction of a suitable plurisubharmonic function 
$\rho$ near the origin, such that the sublevel sets of $\rho$ are strongly pseudoconvex and admit 
strong deformation retraction to $M\cup N$. We also give the application of this result to totally real immersions of real $n$-manifolds in 
$\mathbb{C}^n$ with only finitely many double points, and such that the union of the tangent spaces at each intersection in some local coordinates coincides with 
$M\cup N$, described above.
\end{abstract}


\maketitle

\section{Introduction}

Many classical problems in complex analysis are solvable on Stein manifolds (see for instance \cite{lit21}). It is therefore a very useful property for a subset of a manifold to have open Stein neighborhoods. However, to solve certain problems some further suitable properties of such neighborhoods are needed. For sets that either have tubular neighborhoods or allow uniformly $H$-convex neighborhoods one obtains holomorphic approximation theorems 
(see e.g. Nirenberg and Wells \cite{lit27},
Chirka \cite{Cir69}).


It is also important to control the homotopy type or the shape of the neighborhoods, and hence having 
the so-called {\em regular} neighborhoods; these are neighborhoods which admit a strong deformation retraction 
to a given set. By the results of Forstneri\v{c} \cite[Theorem 2.2]{lit1n} (see also \cite{F92}) and 
Slapar \cite{lit2n} every closed real surface which is smoothly immersed into a complex surface has 
a basis of regular Stein neighborhoods, provided that there are only finitely many double points and 
only hyperbolic complex points, and they are all of special type. Near a special double point the surface 
is given as a model case of a union of two totally real planes in $\mathbb{C}^2$, intersecting only at the origin. 
Proposition \ref{trdi} and \cite[Proposition 4.3]{Tadej2} by the author further extend the above result, but 
still only for some special cases.

In this paper we consider the generalization to higher dimensions, to the case of a union of two totally real 
subspaces of maximal dimension in $\mathbb{C}^n$, intersecting only at the origin. Every such union is 
complex-linearly equivalent to $M(A)\cup N= (A+iI)\mathbb{R}^n\cup \mathbb{R}^n$, where $A$ is a real 
matrix determined up to real conjugacy and such that $i$ is not an eigenvalue of $A$. The problem is to find 
a suitable plurisubharmonic function $\rho$ near the intersection. By a result of Weinstock \cite{Wein} each 
compact subset of $M(A)\cup N$ is polynomially convex if and only if $A$ has no purely imaginary eigenvalue of 
modulus greater than one. In this case one can easily obtain a non-negative plurisubharmonic function, vanishing on 
$M(A)\cup N$.
However, to construct regular neighborhoods additional hypoteses on the gradient $\nabla \rho$ are needed. 
For this reason, as in \cite{Tadej2} we prefer to work with functions, depending polynomially in squared 
Euclidean distance functions to $M$ and $N$ respectively. Furthermore, we are now able to give more streamlined 
computations concerning the Levi form of $\rho$. This enables us to prove the existence of regular Stein 
neighborhoods of $M(A)\cup N$, provided that the eigenvalues of $A$ are sufficiently close to zero; 
see Theorem \ref{izrek} for an estimate of how small these eigenvalues can be. At the end we also give the application of this result to totally real immersions of real $n$-manifolds in 
$\mathbb{C}^n$ with only finitely many double points, and such that the union of the tangent spaces at each intersection in some local coordinates coincides with 
$M(A)\cup N$, described above. 
In connection to this we also note that Weinstock's result has been recently generalised by Gorai \cite{Gorai} and Shafikov and Sukhov \cite[Theorems 1.3 and 4.2]{sukhov}, to the effect that a union of two maximally totally real submanifolds in $\mathbb{C}^n$, intersecting transversally at the origin, is polynomially convex near the origin, if the union of their tangent spaces at the origin is polinomially convex near the origin. 

%
%
%
%
\section{The Euclidean distance function to a totally real subspace}\label{DTR}

Throughout this paper $z=(z_1,\ldots,z_n)$ will be standard holomorphic coordinates and $(x,y)=(x_1,\ldots,x_n,y_1,\ldots,y_n)$ corresponding real coordinates on $\mathbb{C}^n= (\mathbb{R}+ i\mathbb{R})^n\approx \mathbb{R}^{2n}$ with respect to $z_j=x_j+iy_j$ for all $j\in\{1,\ldots,n\}$.

By $\langle\cdot,\cdot\rangle$ and $|\cdot|$ respectively we denote the Euclidean inner product and Euclidean 
distance on any $\mathbb{C}^r$, $r\in \mathbb{N}$:
\begin{align*}
&\langle\xi,\eta\rangle=\sum_{j=1}^r\xi_j\overline{\eta}_j,\qquad \xi=(\xi _1,\ldots,\xi _r),\quad \eta=(\eta _1,\ldots,\eta _r),\\
&|\xi|=\sqrt{\langle\xi,\xi\rangle}
\end{align*}
In real notation with $\xi_j=s_j+it_j$ and $\eta_j=u_j+v_j$ for $j\in\{1,\ldots,r\}$ we have
\[
\langle\xi,\eta\rangle=\sum_{j=1}^r( s_j u_j+ t_j v_j),\quad \xi=(s_1,\ldots,s_r,t_1,\ldots,t_r),\,\eta=(u_1,\ldots,u_r,v_1,\ldots,v_r).
\]

Next, recall that a real linear subspace in $\mathbb{C}^n$ is called {\em totally real} if it contains no complex subspace. 
Let now $M$ and $N$ be linear totally real subspaces of maximal dimension $n$ in $\mathbb{C}^n$, intersecting at the origin. It is not difficult to prove (see e.g. \cite{Wein}) that there exists a non-singular complex linear transformation which maps $N$ onto $\mathbb{R}^n\approx (\mathbb{R}+ i0)^n\subset \mathbb{C}^n$ and $M$ onto $M(A)=(A + i I)\mathbb{R}^n$, where $A$ is the real Jordan canonical form, i.e. $A$ is a square block matrix, having zero-matrices as off-diagonal blocks, and each of the main diagonal blocks satisfies one of the two conditions listed below: 
\begin{enumerate}
\item 
\label{case11}A matrix with $a\in \mathbb{R}$ on the main diagonal, possibly with $\delta\in \mathbb{R}\{0\}$ on the upper diagonal, and zeros otherwise, i.e. 
            \begin{equation}\label{realE}
            \left [\begin{array}{c c c c}
                                   a         & \delta            & \;     & \;  \\
				   \;        & a    & \ddots & \;  \\
				   \;        & \;           & \ddots & \delta   \\
                                   \;        & \;           & \;     & a 
                      \end{array}\right], 
            \end{equation}          
	    %
	    %
\item 
\label{case22}A square block matrix, having $2\times 2$ main diagonal blocks with complex eigenvalues, possibly with the $2\times 2$ identity-matrix $I_2$ multiplied by $\delta\in \mathbb{R}\{0\}$ on the upper diagonal, and $2\times 2$ zero-matrices othervise, i.e. 
                            \begin{equation}\label{complexE}
                            \left [\begin{array}{c c c c}
                                                      C    & \delta I_2       & \;     & \;    \\
						      \;     & C     & \ddots & \;    \\     
						      \;     & \;      & \ddots & \delta I_2     \\
                                                      \;     & \;      & \;     & C   
                                   \end{array}\right],\,\,
                                   C=\left[\begin{array}{c c}   
	                                            c & -b\\
	                                            b & c
                                       \end{array}\right],  b,c\in\mathbb{R}, \, c\neq 0,\, c^2+(1-b^2)^2\neq 0,
                            \end{equation}           
                            %
\end{enumerate}
Moreover, $A$ is uniquely determined up to the order of diagonal blocks, and a non-zero real number $\delta$ can be chosen arbitrarily. Also, the degerate case (i.e. $1\times 1$ matrix in case (\ref{realE}) and $2\times 2$ matrix in case (\ref{complexE})) is considered here, though it lacks an upper diagonal or block-upper diagonal, respectively. Clearly, $A$ is diagonalizable if and if all diagonal blocks are degenerate.

For any $\delta\in \mathbb{R}$ we set  
\begin{equation}\label{Adelta}
A_{\delta}=\diag(A_{\delta,1}, \ldots, A_{\delta,\alpha})
\end{equation} 
to be a $n\times n$ block diagonal matrix and such that for every $j\in\{1,\ldots,\alpha\}$ the diagonal block  $A_{\delta,\alpha}\in \mathbb{R}^{n_j\times n_j}$ is of the form (\ref{realE}) or (\ref{complexE}), possibly degenerate, and $n_1+\ldots+n_{\alpha}=n$.

\begin{lemma}\label{lemadelta}
Let $A_{\delta}$ for $\delta\in \mathbb{R}$ be defined as in (\ref{Adelta}), and let $d_{M(A_{\delta})}(x,y)$ be the squared Euclidean distance from $(x,y)=(x_1,\ldots,x_n,y_1,\ldots,y_n)$ to $M(A_{\delta})$ in $\mathbb{R}^{2n}$. If $A_{\delta}$ is non-diagonalizable for $\delta\neq 0$, then 
\[
d_{M(A_{\delta})}(x,y)=d_{M(A_0)}(x,y)+q_{\delta}(x,y),
\]
where $q_{\delta}$ is a homogeneous polynomial of degree $2$ in $x,y$ and such that its coefficients are rational functions in $\delta$ and they have a zero at $\delta=0$. Furthermore, if $A_0=\diag (C_{1},C_{2}\ldots,C_{\beta},a_1,\ldots a_{\gamma})$, $2\beta+\gamma=n$, where 
                                                $C_j=\left[\begin{array}{c c}   
	                                            c_j & -b_j\\
	                                            b_j & c_j
                                       \end{array}\right]$, and $c_{1},b_{1}\ldots,c_{\beta},b_{\beta},a_1,\ldots a_{\gamma}\in \mathbb{R}$, then
\begin{align}\label{dM0}
d_{M(A_0)}(x,y)=&\sum_{j=1}^{\beta}\frac{(x_{2j-1}-c_jy_{2j-1}+b_jy_{2j})^2+(x_{2j}-c_jy_{2j}-b_jy_{2j-1})^2}{1+c_j^2+b_j^2}\\
&+\sum_{j=2\beta+1}^{2\beta+\gamma}\frac{(x_j-a_jy_j)^2}{1+a_j^2}.\nonumber
\end{align}
In particular, if $A_0=\diag (C_{1},C_{2}\ldots,C_{n/2})$ (respectively $A_0=\diag (a_1,\ldots a_{n})$), then $d_{M(A_0)}(x,y)$ is equal to the first (respectively second) term of the above sum (\ref{dM0}) for $\beta=n/2$ (respectively $\beta=0$, $\gamma=n$).
\end{lemma}

\begin{proof}
Let $M(A_{\delta,j})^{\bot}$ be the orthogonal complement of $M(A_{\delta,j})$ in $\mathbb{C}^{n_j}$ with respect to the standard inner product. It gives decompositions of $M(A_{\delta})$ and its orthogonal complement $M(A_{\delta})^{\bot}$ respectively into pairwise orthogonal linear subspaces:
\[
M(A_{\delta})=M_{\delta,1}\oplus \ldots\oplus M_{\delta,\alpha}, \quad
M(A_{\delta})^{\bot}=M_{\delta,1}^{\bot}\oplus \ldots\oplus M_{\delta,\alpha}^{\bot},
\]
where $M_{\delta,}=\{0\}^{n_1+\ldots+n_{j-1}}\times M(A_{\delta,j})\times \{0\}^{n_{j+1}+\ldots+n_{\alpha}}$ and $M_{\delta,j}^{\bot}=\{0\}^{n_1+\ldots+n_{j-1}}\times M(A_{\delta,j})^{\bot}\times \{0\}^{n_{j+1}+\ldots+n_{\alpha}}$ for all $j$. 

If for any linear subspace $M\subset \mathbb{C}^r$ we denote the squared Euclidean distance and squared Euclidean projection to $M$ in $\mathbb{C}^r$ respectively by $d_M$ and $p_M$, we obtain    
$d_{M(A_{\delta})}=p_{M(A_{\delta})^{\bot}}=\sum_{j=1}^{\alpha} p_{M_{\delta,j}^{\bot}}=\sum_{j=1}^{\alpha} d_{M_{\delta,j}}=\sum_{j=1}^{\alpha} d_{M(A_{\delta,j})}$. For this reason, it is sufficient to prove the lemma for the case when $A_{\delta}$ is of the form (\ref{realE}) or (\ref{complexE}).

First, we consider the case when $A_{\delta}\in \mathbb{R}^{n\times n}$ is of the form (\ref{realE}). This implies that $M(A_{\delta})$ is given as a span of $n$ linearly independent vectors
\begin{equation}\label{spanMa}
        M(A_{\delta})=\Span \{f_{j}+ae_{j}+\delta e_{j-1}\}_{2\leq j \leq n}\cup \{ae_1+f_2\},
\end{equation}
where $\{e_1,f_1,\ldots,e_{n},f_{2n}\}$ is the standard ortho-normal basis of $\mathbb{R}^{2n}$. 
We observe further that the orthogonal complement is then equal to 
\begin{equation}\label{spanMaC}
        M(A_{\delta})^{\bot}=\Span\{e_{j}-af_{j}-\delta f_{j+1}\}_{1\leq j \leq n-1}\cup  \{e_{n}-af_{n}\}
\end{equation}
Indeed, since every $e_j$ for $j\in \{1,\ldots,n\}$ is orthogonal to all but one vector in the span (\ref{spanMaC}), the span contains $n$ linearly independent vectors.

To simplify the computations we denote the vectors in (\ref{spanMaC}) by
\[
        g_j=e_{j}-af_{j}-\delta f_{j+1}, \,\, 1\leq j \leq n-1, \qquad g_n=e_{n}-af_{n}
\]
and we perform Gram-Schmidt process $g_{m}^{'}=g_m-\sum_{j<m}\frac{\langle g_m,g_j^{'}\rangle}{|g_j^{'}|^2}g_j^{'}$ to obtain orthogonal basis $\{g_1^{'},\ldots,g_n^{'}\}$ of $M(A_{\delta})^{\bot}$. We show by induction that  
\begin{equation}\label{rac}
        g_j^{'}=e_{j}-af_{j}-\delta v_{j}, \qquad j\in \{1,\ldots,n\},
\end{equation}
where the components of $v_j$ are rational functions in $\delta$ and without a pole at $\delta=0$. Suppose (\ref{rac}) holds for all $1\leq j<m$. For any $j<m$ it then follows that $|g_j^{'}|^2$ is a rational function in $\delta$ and it has no zero at $\delta=0$, and $\langle g_m,g_j^{'}\rangle$ is a rational function in $\delta$ and without a pole at $\delta=0$. This immediately implies (\ref{rac}).

The squared Euclidean distance of $(x,y)$ to $M(A_{\delta})$ is thus
\begin{align*}
d_{M(A_{\delta})}(x,y)=&\sum_{j=1}^{n}\frac{\big\langle g^{'}_j,(x,y)\big\rangle^2}{|g_j^{'}|^2}=\sum_{j=1}^{n}\frac{\Big(\big\langle e_{j}-af_{j},(x,y)\big\rangle +\delta\big\langle v_j,(x,y)\big\rangle\Big)^2}{|e_{j}-af_{j}-\delta v_{j}|^2}\\
=&\sum_{j=1}^{n}\frac{\big\langle e_{j}-af_{j},(x,y)\big\rangle^2}{|e_{j}-af_{j}|^2}+\delta\sum_{j=1}^{n}\frac{\big\langle v_j,(x,y)\big\rangle\big\langle 2g_j^{'}+3\delta v_j,(x,y)\big\rangle}{|g_j^{'}|^2}\\
&-\delta
\sum_{j=1}^{n}\frac{\big\langle e_{j}-af_{j},(x,y)\big\rangle^2\big\langle 2g_j^{'}+3\delta v_j,v_j\big\rangle}{|e_{j}-af_{j}|^2|g_j^{'}|^2}.
\end{align*}
Observe that the sums in the last are homogeneous polynomials of degree $2$ in $x,y$ and such that its coefficients are rational functions in $\delta$ and in addition they have no a pole at $\delta=0$. Further, for $\delta=0$ in (\ref{spanMaC}) we see that $M(A_{0})^{\bot}=\Span\{e_{j}-af_{j}\}_{1\leq j \leq n}$ is a span of orthogonal vectors, hence the first term in the obove sum is equal to 
\[
d_{M(A_{0})}(x,y)=\sum_{j=1}^n\frac{(x_j-ay_j)^2}{1+a^2}.
\]
This completes proof in the case when $A_{\delta}$ is of the form (\ref{realE}).

In a similar fashin we shall now deal with the case when $A_{\delta}$ is of the form (\ref{complexE}) and thus $n$ even. We have 
\begin{align*}
        M(A_{\delta})=&\Span \{ce_{j}+f_{j}+be_{j+1}+\delta e_{j-2},-be_{j}+f_{j+1}+c f_{j+1}+\delta e_{j-1}\}_{ j \in \{3,5,\ldots, n-1\}}\\
        &\cup \{ce_1+f_1+be_2,-be_1+f_2+ce_2\}
\end{align*}
and
\begin{align}\label{M2}
        M(A_{\delta})^{\bot}=&\Span \{e_{j}-cf_{j}+bf_{j+1}-\delta f_{j+2},e_{j+1}-bf_{j}-c f_{j+1}-\delta e_{j+3}\}_{ j \in \{1,3,\ldots, n-3\}}\nonumber\\
        &\cup \{-cf_{n-1}+e_{n-1}+bf_n,-bf_{n-1}+e_n-cf_n\}.
\end{align}
Again, every $e_j$ for $j\in \{1,\ldots,n\}$ is orthogonal to all but one vector in the span (\ref{M2}), hence the vectors in the span are linearly independent.

Next, we set
\begin{align*}
        &h_{j}=-cf_{j}+e_{j}+bf_{j+1}-\delta f_{j+2},\quad j \in \{1,3,\ldots, n-3\}\\
        & h_{j}=-bf_{j-1}+e_{j}-c f_{j}-\delta e_{j+3},\quad j \in \{2,4,\ldots, n-2\}\\
        &h_{n-1}=-cf_{n-1}+e_{n-1}+bf_n, \quad h_{2n}=-bf_{n-1}+e_n-cf_n
\end{align*}
and proceed with the Gram-Schmidt process to obtain orthogonal basis $\{h_1^{'},\ldots,h_n^{'}\}$ of $M(A_{\delta})^{\bot}$. Similar to (\ref{rac}), we now have
\begin{align}\label{induC}
 h_{j}^{'}=&-cf_{j}+e_{j}+bf_{j+1}-\delta w_{j},\quad j \in \{1,3,\ldots, n-1\}\\
         h_{j}^{'}=&-bf_{j-1}+e_{j}-c f_{j}-\delta w_{j},\quad j \in \{2,4,\ldots, n\}\nonumber
\end{align}
where the components of $w_j$ are rational functions in $\delta$ and without a pole at $\delta=0$.

It follows that
\begin{align*}
d_{M(A_{\delta})}(x,y)=&\sum_{j=1,3,\ldots,n-1}\frac{\big\langle-cf_{j}+e_{j}+bf_{j+1},(x,y)\big\rangle^2}{|-cf_{j}+e_{j}+bf_{j+1}|^2}\\
&+\sum_{j=2,4,\ldots,n}\frac{\big\langle-bf_{j-1}+e_{j}-c f_{j},(x,y)\big\rangle^2}{|-bf_{j-1}+e_{j}-c f_{j}|^2}\\
&-\delta
\sum_{j=1,3,\ldots,n-1}\frac{\big\langle-cf_{j}+e_{j}+bf_{j+1},(x,y)\big\rangle^2\big\langle2h_j^{'}+3\delta w_j,w_j\big\rangle}{|-cf_{j}+e_{j}+bf_{j+1}|^2|h_j^{'}|^2}\\
&-\delta
\sum_{j=2,4,\ldots,n}\frac{\big\langle-bf_{j-1}+e_{j}-c f_{j},(x,y)\big\rangle^2\big\langle 2h_j^{'}+3\delta w_j,w_j\big\rangle}{|-bf_{j-1}+e_{j}-c f_{j}|^2|h_j^{'}|^2}\\
&+\delta\sum_{j=1}^{n}\frac{\big\langle w_j,(x,y)\big\rangle\big\langle 2h_j^{'}+3\delta w_j,(x,y)\big\rangle}{|h_j^{'}|^2}
\end{align*}
Again, we observe that the sums in the last three terms are polynomials of degree $2$ in $x,y$ and such that its coefficients are rational functions in $\delta$ and in addition they have no pole at $\delta=0$. Since (\ref{induC}) for $\delta=0$ is a span of orthogonal vectors, the first two terms in the obove sum are equal to 
\begin{align*}
d_{M(A_{0})}(x,y)=&\sum_{j=1,3,\ldots,n-1}\frac{(x_j-cy_j+by_{j+1})^2}{1+c^2+b^2}+\sum_{j=2,4,\ldots,n}\frac{(x_j-cy_j-by_{j-1})^2}{1+c^2+b^2}\\
=&\sum_{j=1}^n\left(\frac{(x_{2j-1}-cy_{2j-1}+by_{2j})^2}{1+c^2+b^2}+\frac{(x_{2j}-cy_{2j}-by_{2j-1})^2}{1+c^2+b^2}\right)
\end{align*}
This completes proof of the lemma.
\end{proof}

\section{Local construction at the intersection}

Given a $\mathcal{C}^2$-function $f\colon \Omega\to \mathbb{R}$ on an open set $\Omega\subset\mathbb{C}^n$ we denote 
holomorphic and antiholomorphic derivatives of $f$ by
$
\frac{\partial f}{\partial z_j}=\frac{1}{2}\Big(\frac{\partial f}{\partial x_j}-i\frac{\partial f}{\partial y_j}\Big), \, \frac{\partial f}{\partial \overline{z}_j}=\frac{1}{2}\Big(\frac{\partial u}{\partial x_j}+i\frac{\partial u}{\partial y_j}\Big).
$
For $1\leq r \leq n$ we further introduce the notation  
\[
\Big(\frac{\partial f}{\partial z}\Big)_r=
\Big(\frac{\partial f}{\partial z_1,},\ldots,\frac{\partial f}{\partial z_r}\Big), 
\qquad
\Big(\frac{\partial f}{\partial \bar{z}}\Big)_r=
\Big(\frac{\partial f}{\partial \bar{z}_1,},\ldots,\frac{\partial f}{\partial \bar{z}_r}\Big), 
\]

The {\em Levi form} is defined by
\[
   \mathcal{L}_{(z)}(f;\xi)= \sum_{j,k =1}^n \frac{\partial^2 f}{\partial z_j\partial\bar{z}_k }(z)
\, \xi_j\overline{\xi}_k, \quad \xi=(\xi_1,\ldots, \xi_n)\in \mathbb{C}^n.
\]
A function $u$ is {\em strictly plurisubharmonic} if and only if 
$\mathcal{L}_{(z)}(f;\cdot)$ is a positive definite Hermitian quadratic form at each point $z\in \Omega$, 
and this is the case precisely when all leading principal minors of the complex 
Hessian matrix of $f$ are positive on $\Omega$, i.e.
the determinant of
\[
H_r^{\mathbb{C}}(f) =\left[ \frac{\partial f}{\partial z_j \partial \overline{z}_k}  \right]_{j,k=1}^r 
\]
is positive for every $r\in \{1,2,\ldots,n\}$. In particular, $H_n^{\mathbb{C}}(f)$ is the complex Hessian matrix of $f$.

To simplify the computations of the complex Hessian matrices or its determinants in the procedings of this section, 
let us also introduce the following matrix notation:
\begin{align}\label{noteM}
&\xi\, \eta^T=\sum_{j=1}^r\xi _j\eta _j,\qquad \xi=(\xi _1,\ldots,\xi _r),\quad \eta=(\eta _1,\ldots,\eta _r),\\
&\xi^T \eta=\left[ \xi _j \eta _k \right]_{j,k=1}^r. \nonumber
\end{align}

Before starting with the complex Hessians of functions, which are related to totally real subspaces, we prove a simple fact on the determinant of a sum of a matrix a some matrices of rank one. It is stated in the above notation.

\begin{lemma}\label{detlema}
Let $u_1,\ldots,u_s,v_1,\ldots,v_s \in \mathbb{C}^2$, where $s\in \mathbb{N}$, and let $B\in \mathbb{C}^{2\times 2}$. Then 
\begin{align*}
       \det &\big(B+\sum_{l=1}^s u_l v_l^T\big)=\det (B)+\Trace (B)\sum_{l=1}^s v^T_lu_l-\sum_{l=1}^sv_l^TBu_l +\det\big(\sum_{l=1}^s u_l v_l^T\big),\\
                                       \det &(\sum_{l=1}^s u_l v_l^T)=\frac{1}{2}\sum_{l,m=1}^s
                                           (u_l^Tv_l u_m^ Tv_m-u_l^Tv_m u_m^Tv_l).\nonumber                                          
\end{align*}
\end{lemma}

\begin{proof}
We prove the lemma by induction on $s$. The case $s=1$ is a simple matrix identity
\begin{equation}\label{detizrazpomoc}
\det (B+uv^T)=\det (B) +\Trace (B)u^T v-v^TBu, \qquad B\in \mathbb{C}^{2\times 2},\qquad u,v\in \mathbb{C}^2.
\end{equation}

Next, we assume that the statement holds for some $s\in \mathbb{N}$. Applying (\ref{detizrazpomoc}), we then get
\begin{align*}
\det \big(B&+\sum_{l=1} ^{s+1}u_l v_l ^T\big)=\det (B+\sum_{l=1} ^{s}u_l v_l ^T)+\Trace (B+\sum_{l=1} ^{s}u_l v_l ^T)u_{s+1}^T v_{s}\\
&\qquad \qquad \qquad -v_{s+1}^T\big(B+\sum_{l=1} ^{s}u_l v_l ^T\big)u_{s+1}\\
=&\det (B)+\Trace (B)\sum_{l=1}^s v^T_lu_l-\sum_{l=1}^sv_l^TBu_l +\frac{1}{2}\sum_{l,m=1}^s (u_l^Tv_l u_m^ Tv_m-u_l^Tv_m u_m^Tv_l)\\
&+\Trace (B)u_{s+1}^T v_{s}+\sum_{l=1} ^{s}u_l^T v_l)u_{s+1}^T v_{s}-v_{s+1}^TBu_{s+1}-v_{s+1}^T\sum_{l=1} ^{s}u_l v_l ^Tu_{s+1}
\end{align*}
Regrouping the like terms now easily concludes the proof.
\end{proof}

As in \cite{Tadej2} we prefer to work with polynomials in squared Euclidean distance functions to maximal totally real subspaces in $\mathbb{C}^n$.
The following lemma is a preparation for our key result, Lemma \ref{lemaocena}.

\begin{lemma} \label{lemaformula}
Let $A$ be a diagonalizable real $n\times n$ matrix, and let $d_M$ and $d_N$ respectively be the squared 
Euclidean distance functions to $M=(A+iI)\mathbb{R}^n$ and $N=\mathbb{R}^2$. Assume further that 
$P\in \mathbb{R}[u,v]$ is a polynomial in two variables and set 
$\Delta=\frac{1}{2}(\frac{\partial P}{\partial u}{\scriptstyle (d_M,d_N)} +\frac{\partial P}{\partial v}{\scriptstyle (d_M,d_N)})$. 
Then the function
\[
\rho= P(d_M,d_N).
\]
has the following properties:
\begin{enumerate}
\item \label{pr0}
\vspace{-4mm}
\begin{align}\label{formula0}
               \det (H_1^{\mathbb{C}}(\rho))=&\Delta+\tfrac{\partial^2 P}{\partial u^2}{\scriptstyle (d_M,d_N)} |\tfrac{\partial d_M}{\partial z_1}|^2+\tfrac{\partial^2 P}{\partial v^2}{\scriptstyle (d_M,d_N)}  |\tfrac{\partial d_N}{\partial z_1}|^2\\\nonumber
                                             &+2\tfrac{\partial^2 P}{\partial u \partial v}{\scriptstyle (d_M,d_N)} 
                                                      \Rea \big\langle\tfrac{\partial d_M}{\partial z_1},\tfrac{\partial d_N}{\partial z_1}\big\rangle.
                                                      \end{align}
\item \label{pr1} If $A$ is a real diagonal matrix, then for every $r\in \{2,\ldots,n\}$ it follows that
\begin{align}\label{formula1}
               \det (H_r^{\mathbb{C}}(\rho))=&\Delta^r+\Delta^{r-1}(\tfrac{\partial^2 P}{\partial u^2}{\scriptstyle (d_M,d_N)} |(\tfrac{\partial d_M}{\partial z})_r|^2+\tfrac{\partial^2 P}{\partial v^2}{\scriptstyle (d_M,d_N)}  |(\tfrac{\partial d_N}{\partial z})_r|^2)\\\nonumber
                                             &+2\Delta^{r-1}\tfrac{\partial^2 P}{\partial u \partial v}{\scriptstyle (d_M,d_N)} 
                                                      \Rea \big\langle(\tfrac{\partial d_M}{\partial z})_r,\left(\tfrac{\partial d_N}{\partial z}\right)_r\big\rangle\\
                                             &+\Delta^{r-2}\det (H^{\mathbb{R}} (P)_{(d_M,d_N)})\sum_{\scriptscriptstyle {1\leq j<k\leq r}}Z_{jk}(A),\nonumber
\end{align}
where $Z_{jk}(A)=\Big(\sum_{m=j,k}|\tfrac{\partial d_M}{\partial z_m}|^2\sum_{m=j,k}|\tfrac{\partial d_N}{\partial z_m}|^2-|\sum_{m=j,k}\tfrac{\partial d_M}{\partial z_m}\tfrac{\partial d_N}{\partial \overline{z}_m}|^2\Big)$ and $H^{\mathbb{R}}(P)$ is the real Hessian of $P$. Furthermore, if $r=n=2$, we have $|(\tfrac{\partial d_M}{\partial z})_2|^2=d_M$, $|(\tfrac{\partial d_N}{\partial z})_2|^2=d_N$ and $Z_{12}(A)=d_M d_N-\Big|\sum_{j=1,2}\tfrac{\partial d_M}{\partial z_j}\tfrac{\partial d_N}{\partial \overline{z}_j}\Big|^2$.
\item \label{pr2}If $n=2$ and $A=\left[\begin{array}{c c}   
	                                            c & -b\\
	                                            b & c
                                       \end{array}\right]$ with $b,c\in\mathbb{R}$, then
\begin{align}\label{formula3}
\nonumber\det (H^{\mathbb{C}}(\rho))=&\Delta(\tfrac{\partial^2 P}{\partial u^2}{\scriptstyle (d_M,d_N)} d_M+\tfrac{\partial^2 P}{\partial v^2}{\scriptstyle (d_M,d_N)} d_N+2\tfrac{\partial^2 P}{\partial u \partial v}{\scriptstyle (d_M,d_N)} \Rea (\sum_{j=1,2}\tfrac{\partial d_M}{\partial z_j}\tfrac{\partial d_N}{\partial \overline{z}_j}))\\
&+\Delta^2+\det (H_2^{\mathbb{R}} (P)_{(d_M,d_N)})\Big(d_M d_N-\big|\sum_{j=1,2}\tfrac{\partial d_M}{\partial z_j}\tfrac{\partial d_N}{\partial \overline{z}_j}\big|^2\Big)\\
&-(\tfrac{b}{1+b^2+c^2}\tfrac{\partial P}{\partial u}{\scriptstyle (d_M,d_N)})^2-\tfrac{2b^2}{1+b^2+c^2}d_M\tfrac{\partial^2 P}{\partial u^2}{\scriptstyle (d_M,d_N)}\tfrac{\partial P}{\partial u}{\scriptstyle (d_M,d_N)}.\nonumber
\end{align} 
\end{enumerate}
\end{lemma}

\begin{proof}
A straightforward computation yields
\begin{align*}
\frac{\partial^2 \rho}{\partial z_j \partial \overline{z}_k}= &
\frac{\partial^2 P}{\partial u^2}(d_M,d_N)\frac{\partial d_M}{\partial z_j}\frac{\partial d_M}{\partial \overline{z}_k}
+\frac{\partial^2 P}{\partial v^2}(d_M,d_N)\frac{\partial d_N}{\partial z_j}\frac{\partial d_N}{\partial \overline{z}_k}\\
&+\frac{\partial P}{\partial u \partial v}(d_M,d_N)
                                        (\frac{\partial d_M}{\partial z_j}\frac{\partial d_N}{\partial \overline{z}_k}+\frac{\partial d_N}{\partial z_j}\frac{\partial d_M}{\partial \overline{z}_k})\\
&+\frac{\partial P^2}{\partial u}(d_M,d_N)\frac{\partial d_M}{\partial z_j \partial \overline{z}_k}
+\frac{\partial^2 P}{\partial v}(d_M,d_N)\frac{\partial d_N}{\partial z_j \partial \overline{z}_k}                                        
\end{align*}
The leading principal submatrix of the complex Hessian matrix of $\rho $, consisting of the first $r$ rows and 
in the first $r$ columns for $r\in \{1,2,\ldots,n\}$, can hence in matrix notation (see (\ref{noteM})) be writen as 
\begin{align}\label{HC}
\nonumber H_r^{\mathbb{C}}(\rho)=&
\frac{\partial^2 P}{\partial u^2}(d_M,d_N) \Big(\frac{\partial d_M}{\partial z}\Big)_r\Big(\frac{\partial d_M}{\partial \overline{z}}\Big)_r^T
+\frac{\partial^2 P}{\partial v^2}(d_M,d_N)\Big(\frac{\partial d_N}{\partial z}\Big)_r\Big(\frac{\partial d_N}{\partial \overline{z}}\Big)_r^T\\
&+\frac{\partial^2 P}{\partial u \partial v}(d_M,d_N) \left(\Big(\frac{\partial d_M}{\partial z}\Big)_r\Big(\frac{\partial d_N}{\partial \overline{z}}\Big)_r^T+\Big(\frac{\partial d_N}{\partial z}\Big)_r\Big(\frac{\partial d_M}{\partial \overline{z}}\Big)_r\right)^T\\
&+\frac{\partial P}{\partial u}(d_M,d_N)\left[ \frac{\partial^2 d_M}{\partial z_j \partial \overline{z}_k}  \right]_{j,k=1}^r
+\frac{\partial P}{\partial v}(d_M,d_N)\left[ \frac{\partial^2 d_N}{\partial z_j \partial \overline{z}_k}  \right]_{j,k=1}^r.  \nonumber 
\end{align}

Since $N=\mathbb{R}^n$ we immediately get 
\begin{equation}\label{dN}
d_N(x,y)=\sum_{j=1}^ny_j^2, \qquad\qquad \frac{\partial d_N}{\partial z_j}(x,y)=-iy_j, \quad j\in \{1,\ldots,n\}.
\end{equation}

We proceed by proving property (\ref{pr1}). In this case $A=\diag(a_1,\ldots,a_n)$ for $a_1,\ldots,a_n\in \mathbb{R}$ is a real diagonal matrix. Lemma \ref{lemadelta} and a simple computation now yield 
\begin{equation}\label{dM}
d_M(x,y)=\sum_{j=1}^{n}\tfrac{x_j-a_jy_j}{1+a_j^2}, \qquad \tfrac{\partial d_M}{\partial z_j}(x,y)=\tfrac{(1+ia_j)(x_j-a_jy_j)}{1+a_j^2}, \, j\in \{1,\ldots, n\}.
\end{equation}
Further, (\ref{dN}) and (\ref{dM}) easily imply
\[
\left[ \frac{\partial^2 d_M}{\partial z_j \partial \overline{z}_k}  \right]_{j,k=1}^r=\left[ \frac{\partial^2 d_N}{\partial z_j \partial \overline{z}_k}  \right]_{j,k=1}^r=\tfrac{1}{2}I_r,
\]
where $I_r$ is the $r\times r$ identity-matrix. For $\Delta=\frac{1}{2}(\frac{\partial P}{\partial u}(d_M,d_N) +\frac{\partial P}{\partial v}(d_M,d_N))$ we have 
\begin{align*}
H_r^{\mathbb{C}}(\rho) =\Delta I_{r}+L_r,
\end{align*}
where
\begin{align}\label{Lr}
L_r=&\frac{\partial^2 P}{\partial u^2}(d_M,d_N) \Big(\frac{\partial d_M}{\partial z}\Big)_r\Big(\frac{\partial d_M}{\partial \overline{z}}\Big)_r^T
+\frac{\partial^2 P}{\partial v^2}(d_M,d_N)\Big(\frac{\partial d_N}{\partial z}\Big)_r\Big(\frac{\partial d_N}{\partial \overline{z}}\Big)_r^T\\\nonumber
&+\frac{\partial^2 P}{\partial u \partial v}(d_M,d_N) \left(\Big(\frac{\partial d_M}{\partial z}\Big)_r\Big(\frac{\partial d_N}{\partial \overline{z}}\Big)_r^T+\Big(\frac{\partial d_N}{\partial z}\Big)_r\Big(\frac{\partial d_M}{\partial \overline{z}}\Big)\right)_r^T\nonumber
\end{align}

Observe that for $2\leq r\leq n$ the matrix $L_r$ is of rank two, hence its possible minors of order greater than $2$ must be zero. Further, it is a well known fact about characteristic polynomial of the $r\times r$ matrix which states that for any $r\times r$ matrix $B$ we have
\begin{equation}\label{char}
\det (B+\lambda I_r)=\lambda^r+c_{r-1}\lambda^{r-1}+\ldots +c_1\lambda +c_0, \lambda\in \mathbb{R}
\end{equation}
where the coefficient $c_j$ for $j\in \{1,\ldots,r-1\}$ is equal to the sum of all principal minors of the matrix $B$ or order $r-j$. Therefore
\begin{align}\label{detHrCR}
\det (H_r^{\mathbb{C}}(\rho))&=\Delta ^r+\Trace (L_r)\Delta ^{r-1}+\Delta ^{r-2}\sum_{1\leq j<k\leq r} \Minor_{j,k} (L_r)
\end{align}
where $\Minor_{j,k} (L_r)$ is the corresponding principal minor of the matrix $L_r$ with respect to the $j$-th 
and $k$-th row and column; this is the determinant of the submatrix of $L_r$, formed by taking the elements 
in the $j$-th or the $k$-th column, and in the $j$-th or the $k$-th row.

Using matrix notation (\ref{noteM}) we can write
\begin{align}\label{TrL}
      \Trace (L_r)=&\frac{\partial^2 P}{\partial u^2}(d_M,d_N) \Big(\frac{\partial d_M}{\partial z}\Big)_r^T
                   \Big(\frac{\partial d_M}{\partial \overline{z}}\Big)_r
                          +\frac{\partial^2 P}{\partial v^2}(d_M,d_N)\Big(\frac{\partial d_N}{\partial z}\Big)_r^T
                              \Big(\frac{\partial d_N}{\partial \overline{z}}\Big)_r \nonumber\\
                &+\frac{\partial^2 P}{\partial u\partial v}(d_M,d_N) \left(\Big(\frac{\partial d_M}{\partial z}\Big)_r^T     \left(\frac{\partial d_N}{\partial \overline{z}}\right)_r
                           +\Big(\frac{\partial d_N}{\partial z}\Big)_r^T\Big(\frac{\partial d_M}{\partial \overline{z}}\Big)_r\right).
\end{align}
Also, we apply Lemma \ref{detlema} for the case where $B=0$ and the sum of rank-one matrices is equal to $L_2$. By 
regrouping the terms we get
\begin{align}\label{Lrjk}
      &\Minor_{j,k} (L_r)=\big(\tfrac{\partial^2 P}{\partial u^2}{\scriptstyle (d_M,d_N)}\tfrac{\partial^2 P}{\partial v^2}{\scriptstyle (d_M,d_N)} -(\tfrac{\partial^2 P}{\partial u \partial v}{\scriptstyle (d_M,d_N)})^2\big)\\
                 \cdot \,\big((\tfrac{\partial d_M}{\partial z_j},&\tfrac{\partial d_M}{\partial z_k})^T
                  (\tfrac{\partial  d_M}{\partial \overline{z}_j},\tfrac{\partial d_M}{\partial \overline{z}_k})(\tfrac{\partial d_N}{\partial z_j},\tfrac{\partial d_N}{\partial z_k})^T
                  (\tfrac{\partial d_N}{\partial \overline{z}_j},\tfrac{\partial d_N}{\partial \overline{z}_k})-|(\tfrac{\partial d_M}{\partial z_j},\tfrac{\partial d_M}{\partial z_k})^T
                  (\tfrac{\partial d_N}{\partial \overline{z}_j},\tfrac{\partial d_N}{\partial \overline{z}_k})|^2\big)\nonumber
\end{align}
Finally, (\ref{formula1}) now follows immediately from (\ref{detHrCR}), (\ref{TrL}), (\ref{Lrjk}). To conclude the 
proof of (\ref{pr1}), we set $n=r=2$ and using (\ref{dN}), (\ref{dM}), we see that  
\begin{equation}\label{dMabs}
                                       |\left(\tfrac{\partial d_M}{\partial z}\right)_2|^2=\sum_{j=1}^2|\tfrac{\partial d_M}{\partial z_j}|^2=d_M,\qquad
                                       |\left(\tfrac{\partial^2 d_N}{\partial z}\right)_2|^2=\sum_{j=1}^2|\tfrac{\partial d_N}{\partial z_j}|^2=d_N.
\end{equation}

Next, we prove (\ref{pr2}). By (\ref{lemadelta}) we now have 
\begin{equation}\label{dMC}
d_M(x_1,y_1,x_2,y)=\frac{(x_1-cy_1+by_2)^2}{1+b^2+c^2}+\frac{(x_2-cy_2-by_1)^2}{1+b^2+c^2}.
\end{equation}
Next, a simple computation shows that for any $j\in \{1,\ldots, n\}$: 
\begin{align}\label{dMCodvod}
\frac{\partial d_M}{\partial z_1}&=\frac{1+ic}{1+b^2+c^2}(x_1-cy_1+by_2)+\frac{ib}{1+b^2+c^2}(x_2-cy_2-by_1), \\
\frac{\partial d_M}{\partial z_2}&=\frac{1+ic}{1+b^2+c^2}(x_2-cy_2-by_1)-\frac{ib}{1+b^2+c^2}(x_1-cy_1+by_2) \nonumber
\end{align}
and
\[
\left[ \frac{\partial d_M}{\partial z_j \partial \overline{z}_k}  \right]_{j,k=1}^2=\left[\begin{array}{c c}   
	                                            \frac{1}{2} & -\frac{ib}{1+b^2+c^2}\\
	                                            \frac{ib}{1+b^2+c^2} & \frac{1}{2}
                                       \end{array}\right],
\]
Remember that $\Delta=\frac{1}{2}(\frac{\partial P}{\partial u}(d_M,d_N) +\frac{\partial P}{\partial v}(d_M,d_N))$ and by setting  
$\varepsilon=\frac{b}{1+b^2+c^2}\frac{\partial P}{\partial u}(d_M,d_N)$, we have 
\begin{align*}
H^{\mathbb{C}}(\rho) =\Delta I_{2}+i\varepsilon K_2+L_2,
\end{align*}
where $L_2$ is as in (\ref{Lr}) and $K_2=\left[\begin{array}{c c}   
	                                            0 & -1\\
	                                            1 & 0
                                       \end{array}\right]$.

From (\ref{char}) for $B=i\varepsilon K_2+L_2$ and $\lambda=\Delta$, and using the fact $\Trace (i\varepsilon K_2+L_2)=\Trace (L_2)$, it follows that
\begin{align}\label{detHC22}
\det (H_2^{\mathbb{C}}(\rho))&=\Delta ^2+\Delta\Trace (L_2) +\det (i\varepsilon K_2+L_2).
\end{align}
It is easy to see that
\[
\left(\tfrac{\partial d_M}{\partial z}\right)_2^TK_2\left(\tfrac{\partial d_M}{\partial \overline{z}}\right)_2=2ib d_M, \qquad
\left(\tfrac{\partial d_N}{\partial z}\right)_2^TK_2\left(\tfrac{\partial d_N}{\partial \overline{z}}\right)_2=0,
\]
\[
\left(\tfrac{\partial d_N}{\partial z}\right)_2^TK_2\left(\tfrac{\partial d_M}{\partial \overline{z}}\right)_2+\left(\tfrac{\partial d_M}{\partial z}\right)_2^TK_2\left(\tfrac{\partial d_N}{\partial \overline{z}}\right)_2=0,
\]
Thus, applying Lemma \ref{detlema} for $B=K_2$, we obtain 
\[
\det (i\varepsilon K_2+L_2)=-\varepsilon^2-2\tfrac{\partial^2 P}{\scriptstyle (d_M,d_N)} b\varepsilon d_M+\det (L_2).
\]
Since (\ref{dMabs}) holds also in case when $d_M$ of the form 
(\ref{dMC}), then by using (\ref{Lrjk}) for $j=1,k=2$ we obtain that 
\[
\det (L_2)=\Minor_{1,2}(L_2)=\det (H_2^{\mathbb{R}} (P)_{(d_M,d_N)})\Big(d_M d_N-\big|\sum_{j=1,2}\tfrac{\partial d_M}{\partial z_j}\tfrac{\partial d_N}{\partial \overline{z}_j}\big|^2\Big).
\] 
Using also (\ref{TrL}) for $r=2$ and (\ref{detHC22}), we obtain (\ref{pr2}).

Since $\frac{\partial d_M}{\partial z_1 \partial \overline{z}_1}=\frac{1}{2}$ (see (\ref{dM}) and (\ref{dMC})), 
property (\ref{pr0}) follows immediately from (\ref{HC}).
\end{proof}

The following lemma is esential in the proof of Theorem \ref{izrek}, where regular Stein neighborhoods are constructed.

\begin{lemma}\label{lemaocena}
Let $A$ be a real diagonalizable $n\times n$ matrix, and let $d_M$ and $d_N$ respectively be the squared Euclidean distance functions to $M=(A+iI)\mathbb{R}^n$ and $N=\mathbb{R}^n$. If $A$ satisfy one of the conditions
\begin{enumerate}
\item \label{roaa} $A=\diag (a,\ldots,a)$, where $|a|\leq \frac{1}{\sqrt{15}}$ (respectively $|a|\leq \frac{2}{\sqrt{5}}$) for $n\geq2$ (respectively $n=2)$,
\item \label{ro2}  $A=\diag(a_1,a_2)$, where $|a_1|,|a_2|\leq\frac{1}{\sqrt{15}}$, (n=2),
\item \label{roC}  $A=\left[\begin{array}{c c}   
	                                            c & -b\\
	                                            b & c
                                       \end{array}\right]$, where $|c|,|b|\leq \frac{1}{16}$, (n=2),
\end{enumerate}
then there exists a homogeneous polynomial $P\in \mathbb{R}[u,v]$ such that the function
\[
\rho= P(d_M,d_N).
\]
is strictly plurisubharmonic everywhere except at the origin, and such that
\begin{align}\label{pogojP}
&\{P=0\}\cap \mathbb{R}_+^2=(\mathbb{R}_+\times \{0\})\cup (\{0\}\times \mathbb{R}_+),\,\,\\
& \{\nabla P =0\}\cap \mathbb{R}_+^2=\{0\},\quad\tfrac{\partial P}{\partial u}(u,0)=0=\tfrac{\partial P}{\partial v}(0,v),\,\,u,v\in \mathbb{R}_+.\nonumber
\end{align}
\end{lemma}

\begin{proof}
Recall first that $d_N(x,y)=\sum_{j=1}^ny_j^2$ and $\frac{\partial d_N}{\partial z_j}(x,y)=-iy_j$, $j\in \{1,\ldots,n\}$.

When $A=\diag(a_1,\ldots,a_n)$ is a diagonal matrix, then 
$d_{M}(x,y)=\sum_{j=1}^{n}\frac{(x_j-a_jy_j)^2}{1+a_j^2}$ (see Lemma \ref{lemadelta}) with $\tfrac{\partial d_M}{\partial z_j}(x,y)=\tfrac{(1+ia_j)(x_j-a_jy_j)}{1+a_j^2}$ and
\begin{equation}\label{YRe}
\Rea \big\langle(\tfrac{\partial d_M}{\partial z}{\scriptstyle (x,y)})_r,(\tfrac{\partial d_N}{\partial z}{\scriptstyle (x,y)})_r\big\rangle=-\sum_{j=1}^r\tfrac{a_j(x_j-a_jy_j)y_j}{1+a_j^2},\qquad r\in \{1,\ldots,n\}.
\end{equation}
Next, for $r\in \{1,\ldots,n\}$ we introduce the following useful notation 
\begin{align}\label{notMNr}
&d_{M,r}{\scriptstyle (x,y)}=\big|(\tfrac{\partial d_M}{\partial z}{\scriptstyle (x,y)})_r\big|^2=\sum_{j=1}^r\tfrac{(x_j-a_jy_j)^2}{1+a_j^2},\qquad \\%
&d_{N,r}{\scriptstyle (x,y)}=\big|(\tfrac{\partial d_N}{\partial z}{\scriptstyle (x,y)})_r\big|^2=\sum_{j=1}^r y_j^2,\qquad \nonumber\\
&\vartheta_r=\max_{1\leq j\leq r}\tfrac{|a_j|}{\sqrt{1+a_j^2}}.\nonumber
\end{align}
Using the triangle and the Cauchy-Schwarz inequality, respectively, we then get 
\begin{equation}\label{ocenaRe}
\Big|\Rea \big\langle(\tfrac{\partial d_M}{\partial z})_r,(\tfrac{\partial d_N}{\partial z})_r\big\rangle\Big| 
\leq  \vartheta_r (d_{M,r} d_{N,r})^{\frac{1}{2}}.\nonumber 
\end{equation}
We also compute the expression $Z_{jk}$ in Lemma \ref{lemaformula} (\ref{pr1}):
\begin{align*}
   Z_{jk}(x,y)=&\sum_{m=j,k}\big|\tfrac{\partial d_M}{\partial z_m}{\scriptstyle (x,y)}\big|^2\sum_{m=j,k}\big|\tfrac{\partial d_N}{\partial z_m}{\scriptstyle (x,y)}\big|^2-\Big|\sum_{m=j,k}\tfrac{\partial d_M}{\partial z_m}{\scriptstyle (x,y)}\tfrac{\partial d_N}{\partial \overline{z}_m}{\scriptstyle (x,y)}\Big|^2\\
   =& \tfrac{(x_j-a_jy_j)^2y_j^2}{1+a_j^2} -\tfrac{2(1+a_ja_k)(x_j-a_jy_j)(x_k-a_ky_k)y_jy_k}{(1+a_j^2)(1+a_k^2)} + \tfrac{(x_k-a_ky_k)^2y_k^2}{1+a_k^2} .
\end{align*}

Let us now consider the case $A=\diag (a,\ldots,a)\in \mathbb{R}^{n\times n}$. We see that 
\begin{align}\label{vsotaZjk}
\sum_{\scriptscriptstyle {1\leq j<k\leq r}}Z_{jk}(x,y)
=& \sum_{\scriptscriptstyle {1\leq j<k\leq r}}\tfrac{1}{1+a^2}\big((x_j-a y_j)y_j-(x_k-a y_k)y_k\big)^2\\
=&\sum_{j=1}^r\tfrac{(x_j-a y_j)^2}{1+a^2}\sum_{j=1}^ry_j^2-\tfrac{1}{1+a^2}\Big(\sum_{j=1}^r(x_j-a y_j)y_j\Big)^2,\nonumber
\end{align}
where the last equality is obtained by Lagrange's identity (see \cite[p. 38-39]{CS}). Thus
\begin{equation}\label{ocenaZjk}
0\leq \sum_{\scriptscriptstyle {1\leq j<k\leq r}}Z_{jk}\leq d_{M,r} d_{N,r}, \qquad 2\leq r \leq n.
\end{equation}
Next, we apply Lemma \ref{lemaformula} (\ref{pr1}) to $P(u,v)=u^2v+v^2u$ with $\det (H^{\mathbb{R}} (P)_{(u,v)})=-4(u^2+uv +v^2)$. Together with (\ref{notMNr}), (\ref{ocenaRe}), (\ref{ocenaZjk}), we get the lower bound for the determinant of the Hessian matrix of $\rho=P(d_M,d_N)$ on $\mathbb{C}^n\setminus \{0\}$: 
\begin{align}\label{Hro1}
   \nonumber           \Delta^{2-r} \det (H_r^{\mathbb{C}}(\rho))\geq &\Delta^2+\Delta\big(2d_{N} d_{M,r}+2d_M d_{N,r}-4\vartheta_r(d_M+d_N)(d_{M,r} d_{N,r})^{\frac{1}{2}}\big)\\
                                             &-4(d_M^2+d_Md_N+d_N^2)d_{M,r}d_{N,r},
\end{align}
where $\Delta=\frac{1}{2}(\frac{\partial P}{\partial u}{\scriptstyle (d_M,d_N)}+\frac{\partial P}{\partial v}{\scriptstyle (d_M,d_N)})=\tfrac{1}{2}(d_M^2+4d_M d_N+d_N^2)$ and $2\leq r\leq n$. For every $r\in \{1,\ldots,n\}$ let the expression on the right-hand side of the inequality (\ref{Hro1}) be given as $\Psi_{r,\vartheta_r}$. Taking 
$\vartheta_r=\tfrac{1}{4}$ and regrouping the terms of $\Psi_{r,\frac{1}{4}}$ yields
\begin{align}\label{polPr}
\Psi_{r,\frac{1}{4}}=&d_Md_N(2 d_{M,r}d_N+2 d_{N,r}d_M-4d_{M,r}d_{N,r})\nonumber \\
&+(d_M^2+d_N^2)(2d_M d_N+d_M d_{N,r}+d_{M,r}d_N-4d_{M,r}d_{N,r})\nonumber\\
&+d_M^2d_N\big(2d_{N,r}-\tfrac{5}{2}(d_{N,r}d_{M,r})^{\frac{1}{2}}+\tfrac{25}{32}d_{M}\big) \\
&+d_N^2d_M\big(2d_{M,r}-\tfrac{5}{2}(d_{N,r}d_{M,r})^{\frac{1}{2}}+\tfrac{25}{32}d_{N}\big)\nonumber \\
&+\tfrac{1}{8}d_M^3\big(d_M+4(d_{M,r}d_{N,r})^{\frac{1}{2}}+4d_N\big)\nonumber +\tfrac{1}{8}d_N^3\big(d_N+4(d_{M,r}d_{N,r})^{\frac{1}{2}}+4d_M\big)\nonumber \\
&(\tfrac{1}{10}d_M^4-\tfrac{41}{32}d_M^3d_N+\tfrac{9}{2}d_M^2d_N^2-\tfrac{41}{32}d_Md_N^3+\tfrac{1}{10}d_N^4)+\tfrac{1}{40}(d_M^4+d_N^4).\nonumber
\end{align}
Since $d_M\geq d_{M,r}$ and $d_M\geq d_{M,r}$ for all $r$, the first six terms in the above sum are 
non-negative. The eighth term is a symetric homogeneous polynomial in $d_M$ and $d_N$. It is positive 
on $(M\cup N)\setminus \{0\}$, while on $\mathbb{C}^n\setminus (M\cup N)$ we can respectively factor out $d_M^2d_N^2$, setting $W=\frac{d_M}{d_N}+\frac{d_N}{d_M}$, and after regrouping the terms, we obtain 
$d_M^2d_N^2(\frac{1}{10}W^2-\frac{41}{32}W+\frac{43}{10})$, which is positive there. The last term is clearly positive everywhere except at the origin, thus it follows that 
$\Psi_{r,\frac{1}{4}}$ is positive on $\mathbb{C}^n\setminus \{0\}$. As $\Psi_{r,\vartheta_r}$ is a decreasing function with respect to $\theta_r$ and $\vartheta_r=\frac{|a|}{\sqrt{1+a^2}}$ (see (\ref{notMNr}) and remember that $A=\diag(a,\ldots,a)$) increases with respect to $|a|$, we thus have $ \Psi_{r,\vartheta_r}> 0$ on $\mathbb{C}^n\setminus \{0\}$ for $\vartheta_r\leq \frac{1}{4}$ (hence $|a|\leq\tfrac{1}{\sqrt{15}}$), $r\in\{1,\ldots,n\}$. From (\ref{Hro1}), (\ref{polPr}), it immediately follows that $\det (H_r^{\mathbb{C}}(\rho))\geq 0$, with equality precisely at the origin, provided that $2\leq r\leq n$, $|a|\leq \frac{1}{\sqrt{15}}$, and in addition we have
\begin{align*}
\det \big(H_r^{\mathbb{C}}(\rho)\big)(x,y)& \geq \tfrac{1}{40} (d_{M}^4+d_N^4)\big(\tfrac{1}{2}(d_{M}^2+d_{M}d_N+d_N^2)\big)^{r-2}(x,y)\\
&\geq \tfrac{1}{5\cdot 2^{r+1}}(d_{M}^{2r}+d_N^{2r})(x,y) \geq \tfrac{1}{5\cdot 2^{r+1}}\sum_{j=1}^n\Big(\big(\tfrac{(x_j-ay_j)^{2}}{1+a^2}\big)^r+y_j^{2r}\Big).
\end{align*}
Further, by combining Lemma \ref{lemaformula} (\ref{pr0}) and (\ref{notMNr}), (\ref{ocenaRe}), we obtain for $\vartheta_1\leq \frac{1}{4}$ that
\begin{align}\label{HC10}
        H_1^{\mathbb{C}}(\rho) \geq & \Delta+\big(2d_{N} d_{M,1}+2d_M d_{N,1}-4\vartheta_1(d_M+d_N)(d_{M,1} d_{N,1})^{\frac{1}{2}}\big)\nonumber\\
                               \geq & d_M(\tfrac{1}{8}d_M-(d_{M,1} d_{N,1})^{\frac{1}{2}}+d_{N,1})+d_N(\tfrac{1}{8}d_N-(d_{M,1} d_{N,1})^{\frac{1}{2}}+d_{M,1})\\
                               & +\tfrac{3}{8}(d_M^2+d_N^2)+2d_Md_N. \nonumber
\end{align}
and we deduce that $H_1^{\mathbb{C}}(\rho)$ vanishes at the origin and is positive elsewhere for $|a|\leq\tfrac{1}{\sqrt{15}}$ (hence $\vartheta_r\leq \frac{1}{4}$). This concludes the proof that $\rho=d_M^2d_N+d_Md_N^2$ is strictly plu\-ri\-sub\-harmo\-nic precisely on $\mathbb{C}^n\setminus\{0\}$, 
if $M=M(\diag(a,\ldots,a))$, $|a|\leq\tfrac{1}{\sqrt{15}}$.

We proceed with the case $A=\diag(a_1,a_2)$. 
%
%
Using Lemma \ref{lemaformula} (\ref{pr0}),(\ref{pr1}), again applied to $P(u,v)=u^2v+v^2u$ (and $\rho=P(d_M,d_N)$), 
together with the estimates (\ref{notMNr}), (\ref{ocenaRe}), $0\leq Z_{12}\leq d_M d_N$, we get exactly (\ref{Hro1}) for $r=2$ and (\ref{HC10}). 
Much as in the case $A=\diag(a,\ldots,a)$, the condition $|a_1|,|a_2|\leq \tfrac{1}{\sqrt{15}}$ (and hence $\vartheta_1,\vartheta_2\leq \frac{1}{4}$) now implies that $H^{\mathbb{C}}(\rho)$ is positive definite everywhere, except at the origin.

In particular, if $A=\diag (a,a)$ we obtain even better upper bound for $|a|$. Lemma \ref{lemaformula} (\ref{pr1}) applied to any polynomial $P$ (and $\rho=P(d_M,d_N)$), together with (\ref{YRe}), (\ref{notMNr}), (\ref{vsotaZjk}) for $r=2$, and setting $Y=\sum_{j=1}^2\tfrac{(x_j-ay_j)y_j}{\sqrt{1+a^2}}$, $\theta=\frac{a}{\sqrt{1+a^2}}$, then yields
\begin{align}
\det (H_2^{\mathbb{C}}(\rho))=&\Delta(\tfrac{\partial^2 P}{\partial u^2}{\scriptstyle (d_M,d_N)} d_M+\tfrac{\partial^2 P}{\partial v^2}{\scriptstyle (d_M,d_N)} d_N-2\tfrac{\partial^2 P}{\partial u \partial v}{\scriptstyle (d_M,d_N)} \vartheta Y)\\
&+\Delta^2+\det (H^{\mathbb{R}} (P)_{(d_M,d_N)})(d_M d_N-Y^2).\nonumber
\end{align}
For $P(u,v)=u^3v+5u^2v^2+uv^3$, we get $\Delta=\tfrac{1}{2}(d_M^3+13d_M^2d_N+13d_Md_N^2+d_N^3)$, 
\[
\det (H^{\mathbb{R}} (P)_{(d_M,d_N)})=-3(3d_M^4+20d_M^3d_N+94d_M^2d_N^2+20d_Md_N^3+3d_M^4),
\]
and we further have 
\begin{equation}\label{qrs}
\det (H_2^{\mathbb{C}}(\rho))=R(d_M,d_N)Y^2+\vartheta S(d_M,d_N)Y+T(d_M,d_N),
\end{equation}
where
\begin{align*}
R(u,v)=&3(3u^4+20u^3v+94u^2v^2+20uv^3+3v^2),\\
S(u,v)=&-(3u^5+59u^4v+302u^3v^2+302u^2v^3+59uv^4+3v^5),\\
T(u,v)=&\tfrac{1}{4}(u^6+22u^5v+403u^4v^2+44u^3v^3+403u^2v^4+22uv^5+v^6).
\end{align*}
The discriminant of the expression (\ref{qrs}) with respect to $Y$ is for $\vartheta={\frac{2}{3}}$ equal to
\begin{align*}
-\tfrac{1}{9}(& 45u^{10}+906u^9v+25889u^8v^2+127768u^7v^3+582402u^6v^4-207076u^5v^5\\
               &+582402u^4v^6+127768u^3v^7+25889u^2v^8+906uv^9+45v^{10}),
\end{align*}
which is negative everywhere, except for $u=v=0$. If  
$|a|\leq \frac{2}{\sqrt{5}}$ (and hence $\vartheta\leq\frac{2}{3}$), we have $\det (H_2^{\mathbb{C}}(\rho))\geq 0$ with equality precisely at the origin. Next, by combining Lemma \ref{lemaformula} (\ref{pr0}) and (\ref{ocenaRe}) for $r=1$, we deduce that
\begin{align*}
 H_1^{\mathbb{C}}(\rho)\geq & \tfrac{1}{2}(d_M^3+13d_M^2d_N+13d_Md_N^2+d_N^3)+ 2d_N(3d_M+5d_N)d_{M,1} \nonumber\\
&+2d_M(3d_N+5d_M)d_{N,1}-\vartheta (6d_M^2+40d_Md_N+6d_N^2)(d_{M,1}d_{N,1})^{\frac{1}{2}}.
\end{align*}
By regrouping the terms for $\vartheta \leq{\frac{2}{3}}$ ($|a|\leq \frac{2}{\sqrt{5}}$), we further get
\begin{align*}
\nonumber H_1^{\mathbb{C}}(\rho)&\geq \tfrac{4}{3}(d_M \sqrt{d_{N,1}}-d_N\sqrt{d_{M,1}})^2+ d_M^2\big(\tfrac{1}{2}d_M-4 (d_{M,1}d_{N,1})^{\frac{1}{2}}+8d_{N,1}\big)\\
&+ d_N^2\big(\tfrac{1}{2}d_N-4 (d_{N,1}d_{M,1})^{\frac{1}{2}}+8d_{M,1}\big)+6d_Nd_M\big(d_{M}-2(d_{M,1}d_{N,1})^{\frac{1}{2}}+d_{N}\big)\\
& \nonumber+6d_Nd_M(\sqrt{d_{M,1}}-\sqrt{d_{N,1}})^2+\tfrac{2}{3}(d_{M,1}d_N^2+d_{N,1}d_M^2)+\tfrac{1}{2}d_Md_N(d_M+d_N).\nonumber
\end{align*}
Since $d_M\geq d_{M,1}$ and $d_N\geq d_{N,1}$, all the terms on the right hand-side of the equality are non-negative, and in addition the last term vanishes precisely at the origin. This proves that $\rho=d_M^3d_N+5d_M^2d_N^2+d_Md_N^3$ for $M=M\big(\diag (a,a)\big)$ with $|a|\leq \frac{2}{\sqrt{5}}$ is strictly plurisubharmonic precisely on $\mathbb{C}^n\setminus\{0\}$.

Finally, let $n=2$ and assume that $A$ has complex eigenvalues. By (\ref{dM0}) we have $d_M(x_1,y_1,x_2,y_2)=\frac{(x_1-cy_1+by_2)^2}{1+b^2+c^2}+\frac{(x_2-cy_2-by_1)^2}{1+b^2+c^2}$ and the holomorphic derivatives of $d_M$ are of the form (\ref{dMCodvod}). By regrouping the terms we easily see that
\begin{align*}\label{cauch2}
&\,\,\Rea (\tfrac{\partial d_M}{\partial z_1}\tfrac{\partial d_N}{\partial \overline{z}_1})=-\tfrac{1}{1+b^2+c^2}y_1\big(c(x_1-cy_1+by_2)+b(x_2-cy_2-by_1)\big),\\
&
\begin{array}{ll}
\Rea \big(\sum_{j=1,2}\tfrac{\partial d_M}{\partial z_j}\tfrac{\partial d_N}{\partial \overline{z}_j}\big)=&\tfrac{b}{1+b^2+c^2}\big(y_2(x_1-cy_1+by_2)-y_1(x_2-cy_2-by_1)\big) \nonumber\\  
&-\tfrac{c}{1+b^2+c^2}\big(y_1(x_1-cy_1+by_2)+y_2(x_2-cy_2-by_1)\big),
\end{array}\\
&\,\,\Ima (\tfrac{\partial d_N}{\partial z_2}\tfrac{\partial d_M}{\partial \overline{z}_1}- \tfrac{\partial d_N}{\partial z_1}\tfrac{\partial d_M}{\partial \overline{z}_2}  )=\tfrac{1}{1+b^2+c^2}\big(y_1(x_2-cy_2-by_1)-y_2(x_1-cy_1+by_2\big)\nonumber.
\end{align*}
Applying the Cauchy-Schwarz inequality respectively to the first or third expression and to each term of the sum in the second expression, we get
\begin{equation}\label{ocenaCRea2}
\big|\Rea (\tfrac{\partial d_M}{\partial z_1}\tfrac{\partial d_N}{\partial \overline{z}_1})\big|\leq \tfrac{\sqrt{b^2+c^2}}{\sqrt{1+b^2+c^2}}(d_M d_N)^{\frac{1}{2}}
\end{equation}
and
\begin{align}\label{ocenaCRea}
&\big|\Rea (\sum_{j=1,2}\tfrac{\partial d_M}{\partial z_j}\tfrac{\partial d_N}{\partial \overline{z}_j})\big|\leq \tfrac{|b|+|c|}{\sqrt{1+b^2+c^2}}(d_M d_N)^{\frac{1}{2}},\\
&|\Ima (\tfrac{\partial d_N}{\partial z_2}\tfrac{\partial d_M}{\partial \overline{z}_1}- \tfrac{\partial d_N}{\partial z_1}\tfrac{\partial d_M}{\partial \overline{z}_2}  )|\leq\tfrac{1}{1+b^2+c^2}(d_M d_N)^{\frac{1}{2}}.\nonumber
\end{align}
We now apply Lemma \ref{lemaformula} (\ref{pr0}), once more to $\rho=d_M^2d_N+d_N^2d_M$. Using respectively  (\ref{ocenaCRea2}) and the rough estimates $|\tfrac{\partial d_M}{\partial z_1}|^2,|\tfrac{\partial d_M}{\partial z_2}|^2\geq 0$, $\frac{\sqrt{b^2+c^2}}{\sqrt{1+b^2+c^2}}\leq \frac{1}{4}$ for $|b|,|c|\leq \frac{1}{16}$, then after regrouping the terms we obtain ($\Delta=\frac{1}{2}(d_M^2+4d_Md_N+d_N^2)$): 
\begin{align*}
 H_1^{\mathbb{C}}(\rho)\geq & \Delta+\big(2d_{N} |\tfrac{\partial d_M}{\partial z_1}|^2+2d_M |\tfrac{\partial d_N}{\partial z_1}|^2-4\tfrac{\sqrt{b^2+c^2}}{\sqrt{1+b^2+c^2}}(d_M+d_N)(d_{M}d_{N})^{\frac{1}{2}}\big)\nonumber \\
 \geq & \tfrac{1}{2}(d_M^2+4d_M d_N+d_N^2)-(d_M+d_N)(d_{M}d_{N})^{\frac{1}{2}}\nonumber \\
 \geq & \tfrac{1}{2}(d_M+d_N)(\sqrt{d_M}-\sqrt{d_N})^2+d_Md_N.
\end{align*}
It is immediate that $H_1^{\mathbb{C}}(\rho)$ in non-negative and vanishes precisely at the origin. Furthermore, Lemma \ref{lemaformula} (\ref{pr2}) (with $0\leq Z_{12}\leq d_Md_N$) and (\ref{ocenaCRea}) yield
\begin{align*}
   \nonumber         \det (H_2^{\mathbb{C}}(\rho))\geq &\Delta^2-\Gamma^2+ (\Delta-\tfrac{2b}{1+b^2+c^2}\epsilon)(2d_N)d_M+2\Delta d_Md_N \\ &-2\big(\Delta\tfrac{|b|+|c|}{\sqrt{1+b^2+c^2}}(d_{M}d_{N})^{\frac{1}{2}}+|\Gamma| \tfrac{1}{1+b^2+c^2}(d_M d_N)^{\frac{1}{2}} \big)(2d_M+2d_N)\\
                                             &-4(d_M^2+d_Md_N+d_N^2)d_{M}d_{N}\nonumber 
                    ,
\end{align*}
where $\Gamma=\tfrac{b}{1+b^2+c^2}(2d_Md_N+d_N^2)$. Using the estimates $1+b^2+c^2\geq 1$ and $|c|,|b|\leq \frac{1}{16}$ respectively, and each time regrouping the like terms, we further get
\begin{align*}
   \nonumber   \det (H_2^{\mathbb{C}}(\rho))\geq  &\tfrac{1}{4}d_M^4+\tfrac{17}{2}d_M^2d_N^2+\tfrac{1}{4}d_M^4-b^2d_N^2(12d_M^2+8 d_M d_N+d_N^2) \nonumber \\   
    &-2|b|(d_M^2+8d_Md_N+3d_N^2)(d_M+d_N)(d_{M}d_{N})^{\frac{1}{2}}\\
     &-2|c|(d_M^2+4d_Md_N+d_N^2)(d_M+d_N)(d_{M}d_{N})^{\frac{1}{2}}\\
   \geq &\tfrac{1}{2}d_Md_N^2\big(2\sqrt{d_M}-\sqrt{d_N}\big)^2+ \tfrac{1}{128}d_M^2d_N\big(7\sqrt{d_M}-16\sqrt{d_N}\big)^2 \\
&+\tfrac{1}{8}d_M^3\big(\sqrt{d_M}-\sqrt{d_N}\big)^2+\tfrac{1}{8}d_N^3\big(2\sqrt{d_M}-\sqrt{d_N}\big)^2 +\tfrac{67}{128}d_M^3d_N\nonumber \\
&+\tfrac{1}{256}d_M^4+(\tfrac{31}{256}d_M^4-\tfrac{33}{32}d_M^3d_N+\tfrac{285}{64}d_M^2d_N^2-\tfrac{33}{32}d_Md_N^3+\tfrac{31}{256}d_N^4).\nonumber
\end{align*}
The first six terms are non-negative, while the last one is positive everywhere, except at the origin. Indeed, on $\mathbb{C}^2\setminus (M\cup N)$ it can be seen as $d_M^2d_N^2(\tfrac{31}{256}W^2-\frac{33}{32}W+\frac{539}{128})$, $W=\frac{d_M}{d_N}+\frac{d_N}{d_M}$. Hence $\rho$ is strictly plurisubharmonic on $\mathbb{C}\setminus\{0\}$.

To finish the proof of the lemma we observe that, if $P$ is either $P(u,v)=u^2v+uv^2$ or $P(u,v)=u^3v+5u^2v^2+uv^3$, the property (\ref{pogojP}) is clearly satisfied.
\end{proof}

\begin{remark}
The estimates on the entries of $A$ in the lemma are certainly not optimal and might be improved, while on the other hand it is not clear at the moment of this writing, how to obtain a significantly better estimates. The computations quickly get very lengthy if we increase the degree of the polynomial $P$. 
\end{remark}

\section{Regular Stein neighborhoods}\label{baza}

A system of open Stein neighborhoods $\{\Omega_{\epsilon}\}_{\epsilon\in (0,1)}$ of a set $S$ in a complex manifold $X$ is called a {\em regular}, if for every $\epsilon \in (0,1)$ we have 
\begin{enumerate}
\item $\Omega_{\epsilon}=\cup_{t<\epsilon}\Omega_t, \qquad \overline{\Omega}_{\epsilon}=\cap_{t>\epsilon}\Omega_{t}$,
\item $S=\cap_{\epsilon\in (0,1)}\Omega_{\epsilon}$ is a strong deformation retract of every $\Omega_{\epsilon}$ with $\epsilon\in (0,1)$.
\end{enumerate}

For instance, one way to construct such a system of neighborhoods is to find a non-negative function $\rho$, which is strictly plurisubharmonic in some neighborhood of $S$, and such that $S=\{\rho=0\}=\{\nabla\rho=0\}$. Observe that in this case the sublevel sets $\Omega_{\epsilon}=\{\rho<\epsilon \}$ for $\epsilon$ small enough are Stein, and the flow of the negative gradient vector field $-\nabla\rho$ gives us the strong deformation retraction of $\Omega_{\epsilon}$ to $S$. Note that slightly weaker conditions concerning plurisubharmonicity of $\rho$ can work as well (see e.g. Theorem \ref{izrek} or \cite[Theorem 4.1]{Tadej2}).

For the sake of completeness we also recall the following fact about homogeneous polynomials \cite[Lemma 3.2]{Tadej2}, which will be used later on.

\begin{lemma}\label{lemaPR}
Let $Q,R\in \mathbb{R}[x_1,x_2,\ldots,x_m]$ be real homogeneous polynomials in $m$ variables and of even degree $s$. Assume further that $Q$
is vanishing at the origin and is positive elsewhere. Then for any sufficiently small constant $\epsilon_0 >0$, it follows that 
$Q\geq \epsilon_0\cdot |R|$, with equality precisely at the origin.
\end{lemma}

We are now ready to prove the main result.

\begin{theorem}\label{izrek}
Let $A$ be a real $n\times n$ matrix such that $A-iI$ is invertible. Further, let $M(A)=(A+iI)\mathbb{R}^2$. Then the union $M(A)\cup \mathbb{R}^n$ has a regular system of strongly pseudoconvex Stein neighborhoods if any of the following properties are satisfied:
\begin{enumerate}
\item \label{prva} The real parts of the eigenvalues of $A$ are sufficiently close to a real constant $a$ with $|a|\leq \tfrac{1}{\sqrt{15}}$ ($|a|\leq \frac{2}{\sqrt{5}}$ if $n=2$), while the imaginary parts are sufficiently close to zero.
\item \label{druga} $n=2$ and modulii of the real parts of the eigenvalues of $A$ are $\leq \tfrac{1}{\sqrt{15}}$, while the imaginary parts are sufficiently close to zero. 
\item \label{tretja} $n=2$ and the modulii of the real and imaginary parts of the eigenvalues of $A$ are $\leq\tfrac{1}{16}$.
\end{enumerate} 
Moreover, away from the origin the neighborhoods coincide with sublevel sets of the squared Euclidean distance functions to $M$ and $N$ respectively. 
\end{theorem}

It is clear that non-singular linear transformatios map a regular system of Stein neighborhoods into a regular system of Stein neighborhoods. According to the note in Section \ref{DTR}, the general case of the union of two totally real subspaces $M,N$ of maximal dimension, intersecting at the origin, thus reduces to the situation described in the Theorem \ref{izrek}, i.e. $N=\mathbb{R}^2$ and $M=(A+iI)\mathbb{R}^n$, where $i$ is not the eigenvalue of the real $n\times n$ matrix $A$.

\begin{proof}[Proof of the Theorem \ref{izrek}]
Our goal is to construct the function $\rho$, which is strictly plurisubharmonic everywhere, except maybe at the origin, and such that it satisfies the condition $M(A)\cup \mathbb{R}^n=\{\rho=0\}=\{\nabla \rho=0\}$. Clearly, since any real non-singular matrix $R$ maps $M(A)=(A+iI)\mathbb{R}^n$ onto $M(VAV^{-1})=(VAV^{-1}+iI)\mathbb{R}^n$, we deduce that $\sigma=\rho\circ V^{-1}$ with respect to $M(VAV^{-1})$ inherits all the above properties of $\rho$. It is therefore sufficient to consider the case when $A$ is in the Jordan canonical form (\ref{Adelta}), where the parameter $\delta$ can be chosen arbitrarily.

If $A$ satisfies property (\ref{tretja}) then Lemma \ref{lemaocena}(\ref{roC}) immediately implies the existence of the function $\rho$ with the properties listed obove.

Observe that for any real $n \times n$ matrix $B$ and any homogeneous polynomial $P$ of degree $k\geq 2$ in two variables, it follows that $\det ( H_r^{\mathbb{C}}(P{\scriptstyle (d_{M(B)},d_N)}))$, $r\in \{1,\ldots,n\}$, is a homogeneous polynomial of degree $(2k-2)r$ in $x,y$.

Next, let $A=A_{\delta}$ be of the form (\ref{Adelta}), where $\delta$ is to be chosen later.
By Lemma \ref{lemadelta} we have $d_{M(A)}=d_{M(A_0)}+q_{\delta}$, where $q_{\delta}$ is a homogeneous polynomial of degree $2$ in variables $x_1,\ldots,x_n,y_1,\ldots,y_n$ and such that its coefficients are rational functions in $\delta$ and they have no pole at $\delta=0$. For a homogeneous polynomial $P$ (to be chosen) of degree $k\geq 2$ in two variables we obtain
\begin{equation}\label{redA0}
       \det \big( H_r^{\mathbb{C}}(P{\scriptstyle (d_{M(A_{\delta})},d_N)})\big)=\det \big(H_r^{\mathbb{C}}(P{\scriptstyle (d_{M(A_0)},d_N)})\big)+Q_{\delta}, \quad r\in\{1,\ldots,n\},
\end{equation}
where $Q_{\delta}$ is a homogeneous polynomial of degree $(2k-2)r$ in $x,y$, and in addition its coefficients are rational functions in $\delta$ and without a pole at $\delta=0$. 

Let further 
\[
A_0=\diag (D_{1},\ldots,D_{\beta},d_1,\ldots ,d_{\gamma}),
\]
\[
D_j=C_j+\epsilon_jI_2, \,\,j\in\{1,\ldots,\beta\}, \qquad d_k=a_k+\epsilon_{k+\beta},\,\,k\in\{1,\ldots,\gamma\},
\]
where $I_2$ is the $2\times 2$ identity-matrix,  
                                                $C_j=\left[\begin{array}{c c}   
	                                            c_j & -b_j\\
	                                            b_j & c_j
                                       \end{array}\right]$ is a real matrix and $a_1,\ldots a_{\gamma}$, $ \epsilon_1,\ldots,\epsilon_{\beta+\gamma}\in \mathbb{R}$. Setting $B_0=\diag (C_{1},C_{2}\ldots,C_{\beta},a_1,\ldots a_{\gamma})$, we observe that $d_{M(A_0)}=d_{M(B_0)}+s_{\epsilon}$ and
\begin{equation}\label{redEps}
       \det \big( H_r^{\mathbb{C}}(P{\scriptstyle (d_{M(A_{0})},d_N)})\big)=\det \big(H_r^{\mathbb{C}}(P{\scriptstyle (d_{M(B_0)},d_N)})\big)+S_{\epsilon},  \quad r\in\{1,\ldots,n\},
\end{equation}
where $s_{\epsilon}$ and $S_{\epsilon}$ respectively are homogeneous polynomials of degrees $2$ and $2kr$ in $x,y$, and such that their coefficients are polynomials in variables $\epsilon_1,\ldots,\epsilon_{\beta+\gamma}$ without constant term. For any $j\in \{1,\ldots,\beta\}$ we have 
\begin{align*}
&\frac{(x_{2j-1}-c_jy_{2j-1}+b_jy_{2j})^2+(x_{2j}-c_jy_{2j}-b_jy_{2j-1})^2}{1+c_j^2+b_j^2}=\\
=&\frac{(x_{2j-1}-c_jy_{2j-1})^2+(x_{2j}-c_jy_{2j})^2}{1+c_j^2+b_j^2}+b_j^2\frac{(x_{2j-1}-c_jy_{2j-1})^2+(x_{2j}-c_jy_{2j})^2}{(1+c_j^2+b_j^2)(1+c_j^2)}\\
&+b_j\frac{(2(x_{2j-1}-c_jy_{2j-1})+b_jy_{2j})y_{2j}-(2(x_{2j}-c_jy_{2j})-b_jy_{2j-1})y_{2j-1}}{(1+c_j^2+b_j^2)}.
\end{align*}
Using (\ref{dM0}) it thus follows
\[
d_{M(A_0)}(x,y)=d_{M(\Lambda)}(x,y)+\sum_{j=1}^{\beta}t_j(x,y),
\]
where $\Lambda=\diag (c_1,c_1,\ldots,c_{\beta},c_{\beta},a_1,\ldots,a_{\gamma})$, and for every $j$ the polynomial $t_j$ is homogeneous polynomial of degree $2$ in variables $x_{2j-1},x_{2j},y_{2j-1},y_{2j}$, and such that its coefficients are rational functions in $b_j$ and they have a zero at $b_j=0$.
It further implies that
\begin{equation}\label{redD}
        \det \big( H_r^{\mathbb{C}}(P{\scriptstyle (d_{M(A_0)},d_N)})\big)=\det \big(H_r^{\mathbb{C}}(P{\scriptstyle (d_{M(\Lambda)},d_N)})\big)+T_{b}, \quad r\in\{1,\ldots,n\},
\end{equation}
where $T_{b}$ is a homogeneous polynomial of degree $(2k-2)r$ in $x,y$. In addition, the coefficients of $s_b$ are rational functions in $b_1,\ldots,b_{\beta}$, and such that they vanish for $b_1=\ldots=b_{\beta}=0$.

Observe that $A_{\delta}$ satisfies the condition (\ref{prva}) (respectively (\ref{druga})) of Lemma \ref{izrek} precisely when $\Lambda$ satisfies the condition (\ref{roaa}) (respectively (\ref{ro2})) of Lemma \ref{lemaocena}, provided that constants $b_1,\ldots,b_{\beta}$, $\epsilon_1,\ldots,\epsilon_{\beta+\gamma}$, $\delta$ are small enough. Furthermore, if $\Lambda$ satisfies any of the conditions (\ref{roaa}) or (\ref{ro2}) in Lemma \ref{lemaocena}, then there exists a homogeneous polynomial $P$, and such that $\rho=P(d_{M(\Lambda)},d_N)$ is a polynomial in $x,y$, which is strictly plurisubharmonic everywhere except at the origin, and such that (\ref{pogojP}) is satisfied. We now use Lemma \ref{lemaPR} to see that for sufficiently small constants $b_1,\ldots,b_{\beta}$, $\epsilon_1,\epsilon_{\beta+\gamma}$, $\delta$, and using (\ref{redD}), (\ref{redEps}), (\ref{redA0}), respectively, $\det ( H_r^{\mathbb{C}}(P{\scriptstyle (d_{M(A_{\delta})},d_N)}))$ vanishes at the origin and is positive everywhere else. Since $\nabla \rho=\frac{\partial P}{\partial u}{\scriptstyle (d_{M(A_{\delta})},d_N)}\nabla d_{M(A_{\delta})}+\frac{\partial P}{\partial v}{\scriptstyle (d_{M(A_{\delta})},d_N)}\nabla d_N$, it follows from (\ref{pogojP}) that $M\cup N=\{\rho=0\}=\{\nabla\rho=0\}$.

Finally, mutatis mutandis, the proof given in \cite[Theorem 4.1]{Tadej2} now applies to glue $\rho$ away from the origin with the squared distance functions. We 
choose open balls $B_R$ and $B_{2R}$ respectively, centered at $0$ and with radii $R$ and $2R$, and observe that for any sufficiently small $\epsilon >0$ the sets 
\[
     T_{\epsilon,M}=\{z\in\mathbb{C}^n\setminus \overline{B}_{R}\colon d_M(z)<\epsilon\}, \qquad  
     T_{\epsilon,N}=\{z\in\mathbb{C}^n\setminus\overline{B}_{R}\colon d_N(z)<\epsilon\}
\]
are disjoint. Next, we set $T_{\epsilon}=T_{\epsilon,M}\cup T_{\epsilon,N}$ and define:
\[
      \rho_0(z)=\theta(z)\rho(z)+\bigl(1-\theta(z)\bigr)d_M|_{T_{\epsilon,M}}(z)+\bigl(1-\theta(z)\bigr)d_N|_{T_{\epsilon,N}}(z),\quad z\in B_{2R}\cup T_{\epsilon}.
\]
Here $\theta(z)=\chi\big(\sum_{j=1}^n|z_j|^2\big)$, where $\chi$ is a suitable cut-off function with $\chi(t)=1$ for $t\leq R$ 
and $\chi(t)=0$ for $t\geq 2R$. 

It is clear that $\{\rho_0=0\}=M\cup N$. On 
$(B_{2r}\setminus \overline{B}_R)\setminus (M\cup N)$, but close to $M\cup N$, we have $\nabla\theta$ 
near to tangent directions to $M\cup N$, and $\nabla d_M$ or $\nabla d_N$ respectively are near to 
normal directions to $M$ and $N$. Hence, after possibly choosing $\epsilon$ smaller we obtain $\{\nabla\rho_0=0\}=M\cup N$. The flow of the negative gradient vector field $-\nabla \rho_0$ then yields a deformation retraction of $\Omega_{\epsilon}=\{\rho_0<\epsilon\}$ onto 
$M\cup N$.

Since $\rho$, $d_M$, $d_N$ along with their gradients all vanish 
on $M\cup N$, it follows that for $z\in M\cup N$ and any $\xi\in T_z(\mathbb{C}^n)$ we have 
\[
      \mathcal{L}_{(z)}(\rho_0;\xi)=\theta (z)\mathcal{L}_{(z)}(\rho;\xi)+(1-\theta(z))\mathcal{L}_{(z)}(d_M|_{T_{\epsilon,M}};\xi)
                                                                                 +(1-\theta(z))\mathcal{L}_{(z)}(d_N|_{T_{\epsilon,N}};\xi).
\]
The Levi form of $\rho_0$ is thus positive on $\overline{\Omega}_{\epsilon}\setminus \{0\}$, provided that 
$\epsilon$ is chosen small enough. Further, since the restrictions of plurisubharmonic functions to analytic sets are plurisubharmonic 
and satisfy the maximum principle (see \cite{GRpluri}), there cannot be any compact 
analytic subset of positive dimension in $\mathbb{C}^n$. By a result of Grauert (see \cite[Proposition 5]{lit9}) $\Omega_{\epsilon}$ is then Stein. 
This completes the proof.
\end{proof}


Lemma \ref{lemaocena} can be also applied to prove the existence of regular neighborhoods of certain smooth totally real immersions of a real $n$-manifolds into a complex $n$-manifold; for results on clo\-sed real surfaces immersed into complex surface see \cite[Theorem 2.2]{lit1n}, \cite[Theorem 2]{lit2n}, \cite[Proposition 4.3]{Tadej2}).

\begin{trditev}\label{trdi}
Let $\pi\colon Z\to X$ be an smooth totally real immersion of a closed real $n$-manifold into a complex $n$-manifold $X$, and such that 
$\pi$ has only transverse double points (no multiple points) $q_1,\ldots q_s\in \pi (Z)$ with $\pi^{-1}(q_j)=\{t_j,u_j\}$. For any $j\in \{1,\ldots,s\}$, let the images under tangent map of the tangent spaces of $Z$ at points $t_j$ and $u_j$, respectively, define a union of two totally real subspaces in $T_{q_j}X\approx \mathbb{C}^n$, which is holomorphicaly-equivalent to $(A_j+iI)\mathbb{R}^n\cup \mathbb{R}^n\subset\mathbb{C}^n$, where $A_j$ is a real $n\times n$ matrix with $A_j-iI$ invertible. 
If the entries of $A_j$ for all $j\in\{1,\ldots,s\}$ satisfy any of the conditions (\ref{prva}), (\ref{druga}) or (\ref{tretja}) in Lemma \ref{lemaocena}, then $\tilde{Z}=\pi (Z)$ has a regular Stein neighborhood basis.
\end{trditev}


\begin{proof}
For any double point $q_j$ there exists local holomorphic coordinates $\psi_j\colon U_j\to V_j\subset\mathbb{C}^n$, such that $\psi_j(q_j)=0$ and such that $\psi_j(\tilde{Z})=S_j\cup T_j$, where $S_j$ and $T_j$ are real $n$-manifolds, intersecting only at the origin, and are tangent to $M_j=(A_j+iI)\mathbb{R}^n$ and $N_j=\mathbb{R}^n$ there, respectively.

Next, by following the proof of the local tubular neighborhood (see e.g. \cite[Theorem 3.1]{AmbroSon} or \cite[p. 78-92]{Krantz}) we show that in a sufficiently small  neighborhood of a point $w_0$ on a real $k$-submanifold $S\subset\mathbb{R}^m$ the Taylor expansions of the squared Euclidean distance functions, respectively, to $S$ and to the affine tangent space $M$ to $S$ at $w_0$, agree to the terms of second order. Let $0\in W'\subset\mathbb{R}^r$ and let $F\colon W'\to W$ be a parametrization or $S$ in a neighborhood $W$ of a point $w_0\subset \mathbb{R}^m$, and such that $F(0)=w_0$. By the smoothness of $S$ there exists orthonormal vector fields $(v_1,\ldots,v_{m-k})\colon W\cap S\to \mathbb{R}^m$, spanning the normal space to $S$. We set $\Theta\colon W'\times \mathbb{R}^{m-k}$, $\Theta (\mu,\nu)=F(\mu)+\sum_{j=1}^{m-k}\nu_jv_j(\mu)$. Since the rank of the Jaccobian $J(\Theta)_{(0,0)}$ is maximal, by the implicit mapping theorem $\Theta$ is a smooth diffeomorphism in a small neighborhood of the origin. Let $\Phi(w)=(\mu(w),\nu(w))$ be its smooth inverse, defined in a small neihgborhood $\tilde{W}$ of $w_0$. Observe that the nearest point on $S$ for any $w\in \tilde{W}$ is $F(\mu(w))$ and the squared Euclidean distance from $w$ to $S$ is $d_S(w)=|\nu(w)|=\sum_{j=1}^{m-k}\nu_j^2(w)$. Since $d_S(w)=0$ for $w\in S$ all derivatives in the tangent directions to $S$ vanish at $w_0$, while the derivative of $\nu_j$ in the direction $v_s(w_0)$ is equal to $1$ if $j=s$ and vanishes otherwise. If now $w=(x_1,\ldots,x_m)$ denote coordinates on $\tilde{W}\subset\mathbb{R}^m$ then by Taylor's theorem we have 
\[
\nu_j(w)=\nu_j(w_0)+
\langle v_j(w_0),w-w_0\rangle +\sum_{|\alpha|=3}(w-w_0)^{\alpha}f_{\alpha}(w), 
\]
where $(w-w_0)^{\alpha}=(x_1-x_{0,1}^{\alpha_1}\ldots (x_m-x_{0,m})^{\alpha_m}$ for a multiindex $(\alpha_1,\ldots,\alpha_m)$, and $f_{\alpha}$ is a smooth function for every multiindex $\alpha$. It follows that 
\[
d_S(w)=\sum_{j=1}^{m-k}\langle v_j(w_0),w-w_0\rangle^2+\sum_{|\alpha|=3}(w-w_0)^{\alpha}g_{\alpha}(w)=d_M(w)+\sum_{|\alpha|=3}(w-w_0)^{\alpha}g_{\alpha}(w),
\]
where $d_M$ is a squared Euclidean distance to the affine tangent space to $S$ at $w_0$ and $g_{\alpha}$ is a smooth function for every multiindex $\alpha$.

Observe further that for any homogeneous polynomial of degree $s$ we have 
\begin{align*}
\det H_r^{\mathbb{C}}(P(d_{{\scriptstyle S_j}},d_{T_j})_{(x,y)})=&\det H_r^{\mathbb{C}} (P(d_{M_j},d_{N_j})_{(x,y)})+\sum_{|\alpha|+|\beta|>(2s-2)r}x^{\alpha}y^{\beta}h_{\alpha}(x,y),\\
=&\det H_r^{\mathbb{C}} (P(d_{M_j},d_{N_j})_{(x,y)})+\sum_{|\alpha|+|\beta|=(2s-2)r}x^{\alpha}y^{\beta}H_{\alpha}(x,y),
\end{align*}
where the determinants of the complex Hessians are homogeneous polynomials of degre $(2s-2)r$ in $x,y$ and $h_{\alpha},H_{\alpha}$ are smooth for all multiindices $\alpha$, and in addition $H_{\alpha}(0)=0$.
Lemma \ref{lemaocena} now furnishes a homogeneous polynomial $P$, satisfying (\ref{pogojP}), and such that $P(d_{M_j},d_{N_j})$ is strictly plurisubharmonic ewerywhere, except at the origin. Since $H_{\alpha}(x,y)$ is sufficienty close to zero, provided that $(x,y)$ is close enoug to the origin, by using Lemma \ref{lemaPR} we deduce that $\rho_j=P(d_{M_j},d_{N_j})$ is strictly plurisubharmonic everywhere sufficiently close to the origin, but not at the origin. Moreover, property (\ref{pogojP}) yields $S_j\cup T_j=\{\rho_j=0\}=\{\nabla\rho_j=0\}$ on a sufficiently small neighborhod of the origin.
After possibly shrinking $U_j$, we set $\varphi_j=\rho_j\circ \psi_j\colon U_j\to \mathbb{R}$, and observe that $\varphi_j$ inherits the obove properties from $\rho_j$.

Further, let $\varphi_0=d_{\tilde{Z}}$ and $d_w$ respectively be the squared distance functions 
to $\tilde{Z}$ and to $w\in \tilde{Z}$ in $X$, relative to some Riemannian metric on $X$. It 
is well known that the squared distance function to a smooth totally real submanifold is strictly 
plurisubharmonic in a neighborhood of the submanifold (see \cite[Proposition 4.1]{lit27}). Therefore $\varphi_0$ is strictly plurisubharmonic in some open 
neighborhood $U_0$ of $\tilde{S}\setminus \{p_1,\ldots,p_m\}$.

By standard patching technique we now glue functions $\varphi_j$ for all $j\in \{0,1,\ldots,s\}$ (see e.g. \cite[Theorem 2]{lit2n}). First, we denote 
$U=\cup_{j=0} ^{s} U_j$ and let $p\colon U\to \tilde{Z}$ be a map defined as 
$p(z)=w$ if $d_{\tilde{Z}}(z)=d_{w}(z)$. The map $p$ is well defined and smooth, 
provided that the sets $U_j$ are chosen small enough (see e.g. \cite{Foote}). Next, we choose a partition of 
unity $\{\theta_j\}_{0\leq j\leq s}$ subordinated to $\{U_j\cap \tilde{Z}\}_{0\leq j\leq s}$, 
and such that for every $j\in\{1,\ldots,s\}$ the function $\theta_j$ equals one near the 
point $p_j$. Finally, we define
\[
               \rho(z)=\sum_{j=0} ^s \theta_j\bigl(p(z)\bigr) \varphi_j(z), \qquad z \in U.
\]
We see that $\tilde{Z}=\{\rho=0\}$ and $\nabla\rho(z)=\sum_{j=0}^m \theta_j\bigl(r(z)\bigr) \nabla\varphi_j(z)$ 
for all $z\in U$, thus we further have
\[
        \mathcal{L}_{(z)}(\rho;\xi)=\sum_{j=0} ^s \theta_j(z)\mathcal{L}_{(z)}(\varphi_j;\xi), \qquad 
                                         z\in \tilde{Z}, \quad \xi \in T_p(U).
\]
After shrinking $U$ we obtain that $\{\nabla\rho=0\}=\tilde{Z}$ and $\rho$ is 
strictly plurisubharmonic everywhere, except at the points $q_{1}, \ldots, q_s$. 
Again, the flow of $-\nabla \rho_0$ yields a deformation retraction of $\Omega_{\epsilon}=\{\rho_0<\epsilon\}$ onto 
$\tilde{Z}$, and by \cite[p. 180]{GRpluri} there is no compact 
positive dimensional analytic subset in $\Omega_{\epsilon}$  (there exists a nonconstant plurisubharmonic function $\rho$ on $\Omega_{\epsilon}$). By a result of Grauert (see \cite[Proposition 5]{lit9}) $\Omega_{\epsilon}$ is then Stein.
\end{proof}



\end{document}